\documentclass[12pt]{amsart}
\usepackage[utf8]{inputenc}
\usepackage{color}
\usepackage{float}
\usepackage{times}
\usepackage[hidelinks]{hyperref}
\usepackage{enumerate,latexsym}
\usepackage{latexsym}
\usepackage{subcaption}

\usepackage{fullpage}

\usepackage{comment}

\usepackage{booktabs}

\input{epsf.sty}
\usepackage{graphicx}

\usepackage{amsmath,amsthm,amsfonts,amssymb}
\usepackage{graphicx}
\usepackage{dsfont}

\usepackage{tikz}
\usetikzlibrary{arrows.meta}
\usetikzlibrary{knots}
\usetikzlibrary{hobby}
\usetikzlibrary{arrows,decorations.markings}
\usetikzlibrary{external}
\usetikzlibrary{fadings}

\makeatletter
\def\namedlabel#1#2{\begingroup
 #2%
 \def\@currentlabel{#2}%
 \phantomsection\label{#1}\endgroup
}
\makeatother

\usepackage[lite]{amsrefs}

\renewcommand{\PrintDOI}[1]{\href{http://dx.doi.org/\detokenize{#1}}{doi: \detokenize{#1}}%
	\IfEmptyBibField{pages}{, (to appear in print)}{}}

\theoremstyle{plain}
\newtheorem*{theorem*}{Theorem}
\newtheorem*{thmex*}{Theorem~\ref{example}}
\newtheorem*{thmasymp*}{Theorem~\ref{thmAsymp}}
\newtheorem{theorem}{Theorem}[section]

\newtheorem{corollary}[theorem]{Corollary}

\newtheorem{lemma}[theorem]{Lemma}

\newtheorem{proposition}[theorem]{Proposition}
\newtheorem{remark}[theorem]{Remark}

\newtheorem{example}[theorem]{Example}

\theoremstyle{definition}
\newtheorem{definition}[theorem]{Definition}

\newcommand{\ben}{\begin{enumerate}}
\newcommand{\een}{\end{enumerate}}

\newcommand{\ed}{\end{document}}

\definecolor{rrr}{rgb}{.9,0,.1}

\definecolor{rr}{rgb}{.8,0,.3}

\newcommand\LabGrid[2][purple]

\usepackage[all]{xy}

\usepackage{pdfsync}
\graphicspath{ {./images/} }

\date{}
\title{G-Families of Singquandles}

\title{Oriented Disingquandles and Invariants of Oriented Dichromatic Singular links}

\author[M.I.Sheikh]{Mohd Ibrahim Sheikh}

\address{Department of Humanities And Sciences (Mathematics), St. Peter's Engineering College, Maisammaguda, 
       Medchal, Hyderabad - 500100, India.}
 \email{ibrahim@stpetershyd.com}

 \author[M. Elhamdadi]{Mohamed Elhamdadi}
\address{University of South Florida, Department of Mathematics and Statistics, Tampa, FL, USA}
\email{emohamed@usf.edu}

 \author[D. Ali]{DANISH ALI}
\address{Department of Mathematics, Dalian University of Technology, China}
\email{danishali@mail.dlut.edu.cn}

\date{}
\begin{document}
\maketitle

\begin{abstract}
We introduce and investigate {\it oriented dichromatic singular links}. We also introduce
{\it oriented disingquandles} and use them to define {\it counting invariants for oriented dichromatic singular links}. We provide some examples to show that these invariants distinguish some oriented dichromatic singular links.
\end{abstract}

\tableofcontents

\section{Introduction}
A simple closed curve in $\mathbb{R}^{3}$ is called a mathematical knot and disjoint union of $n$ knots is called an $n$-link or a link with $n$ components. Knots and links are represented by knot and link diagrams in $\mathbb{R}^{2}$ in which we have clear information about the regions where the curve goes over or under itself. Such regions are called crossing points of the knot or link and 
 this is referred to as classical knot theory. This 
theory has been generalized to some other theories such as virtual and singular knot theories. In \cite{SEA} singular links have been generalized to dichromatic singular links in which the components of a singular link are labelled by $``1"$ or $``2"$. 
In \cite{SEA} a notion of unoriented dichromatic singular links was introduced and used singquandles  to construct 
coloring invariants for such dichromatic singular links 
 in order to classify them. 
In this paper we introduce an oriented version of such links and 
 study their counting invariants. More specifically we introduce the notion of {\it oriented dichromatic singular links}.
\par Quandles, singquandles and oriented singquandles are the algebraic structures whose axioms are motivated by the classical Reidemeister moves, generalized singular Reidemeister moves and oriented generalized singular Reidemeister moves respectively \cites{CCEH1, CCEH, EN, J, M, NOS, Oyamaguchi}. 
 Quandles and involutive quandles have been used to distinguish oriented (respectively unoriented) knots and links.  
Involutive quandles when generalized with the purpose of studying unoriented singular links, give rise to the 
algebraic structure called singquandles \cite{CEHN}. 
 Oriented singquandles were introduced in order to study and distinguish oriented singular knots and links \cite{BEHY}.
Considering two or more quandles, one obtains a new algebraic structure, called $\mathbb{G}$-{\it Family of quandles} whose axioms are motivated by handlebody-knot theory \cite{IMJO}. 
The case of $\mathbb{Z}_{2}$-{\it Family of quandles} was investigated by Lee and Sheikh who  showed that the strucure gives an invariant of oriented dichromatic links \cite{LS}. However considering two singquandles, a $\mathbb{Z}_{2}$-{\it Family of singquandles} (or {\it disingquandle}) was obtained and investigated by Sheikh {\it et al.} who showed that the structure gives an invariant of unoriented dichromatic singular links \cite{SEA}.  
\par In this paper, we introduce the notions of {\it oriented dichromatic singular links} and {\it oriented disingquandles} which are natural generalizations of the notions introduced in \cite{SEA}. An $n$-component oriented singular link whose components are labelled either by $``1"$ or $``2"$ forms an {\it oriented dichromatic singular link}. By taking a family of oriented singquandles, we construct an algebraic structure called {\it oriented disingquandle} and whose axioms directly come from the generalized oriented dichromatic singular Reidemeister moves. We investigate properties of {\it oriented disingquandles} and give some examples.  Furthermore, we define and study the counting invariant of oriented dichromatic singular links.  We also distinguish some oriented dichromatic singular links by using oriented disingquandles.
\par This paper is organized in the following manner. In Section~\ref{SOSODL}, we review the notions of singular links, oriented singquandles and oriented dichromatic links. Section~\ref{ODSL} deals with oriented dichromatic singular links. In Section~\ref{GFOSQ}, we introduce oriented disingquandles and study some of their properties thus allowing the construction of examples. Section~\ref{CIUSDL} introduces the coloring invariant of oriented dichromatic singular links and shows that the invariant distinguishes some oriented dichromatic singular links. 

\section{Singular Links, Oriented Singquandles and Oriented Dichromatic Links}\label{SOSODL}

 In $1990$ Vassiliev introduced singular knot theory in \cite{V} as a generalization of classical knot theory.  Axiomatizing the \emph{generalized} Reidemeister moves of singular knot theory led to algebraic structures allowing  generalizations of various classical knot invariants to singular knots.  In \cite{CEHN} and \cite{BEHY}, the authors used involutive quandles and quandles to study unoriented and respectively oriented singular links by introducing {\it singquandles} and {\it oriented singquandles} respectively.  Quandles and singquandles were used in \cite{LS} and in \cite{SEA}  
 to study oriented dichromatic links and unoriented dichromatic singular links by introducing algebraic structures {\it diquandles} and {\it disingquandles} respectively. Now we review the basic definitions of singular links, oriented singquandles and oriented dichromatic links.
 \begin{definition}\cite{V}\label{Def2.1}
  A classical $n$-link $\mathnormal{L}$ in $\mathbb{R}^{3}$ having one or more singular points at its crossing points is called a singular $n$-link and are represented by singular link diagrams $\mathnormal{D}$ in $\mathbb{R}^{2}$.  
 \end{definition}
 A singular link $\mathnormal{L}$ is said to be oriented if its each component is oriented. The following Figure~\ref{OSLinks} shows two $2$-oriented singular link diagrams.
 \begin{figure}[h]
       \centering
		\includegraphics[width=0.6\textwidth]{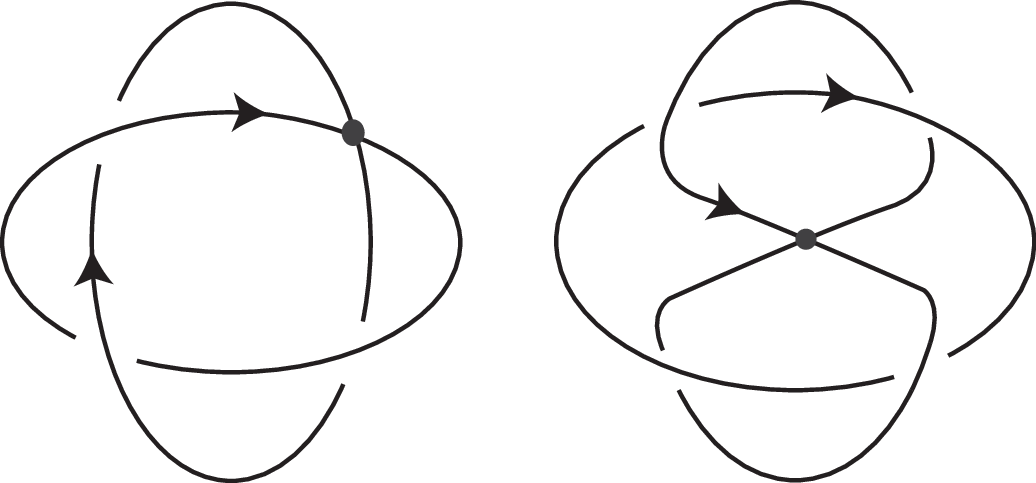}
	\caption{Oriented Singular Links}
	\label{OSLinks}
\end{figure}
\par Under the generating set of oriented Reidemeister moves of singular links in Figure \ref{GSOOSRM}, if two oriented singular link diagrams $\mathnormal{D_1}$ and $\mathnormal{D_2}$ representing the oriented singular links $\mathnormal{L_1}$ and $\mathnormal{L_2}$ respectively, can be obtained from each other, then the links $\mathnormal{L_1}$ and $\mathnormal{L_2}$ are called isotopy equivalent.
\begin{figure}[h]
       \centering
       \includegraphics[width=0.9\textwidth]{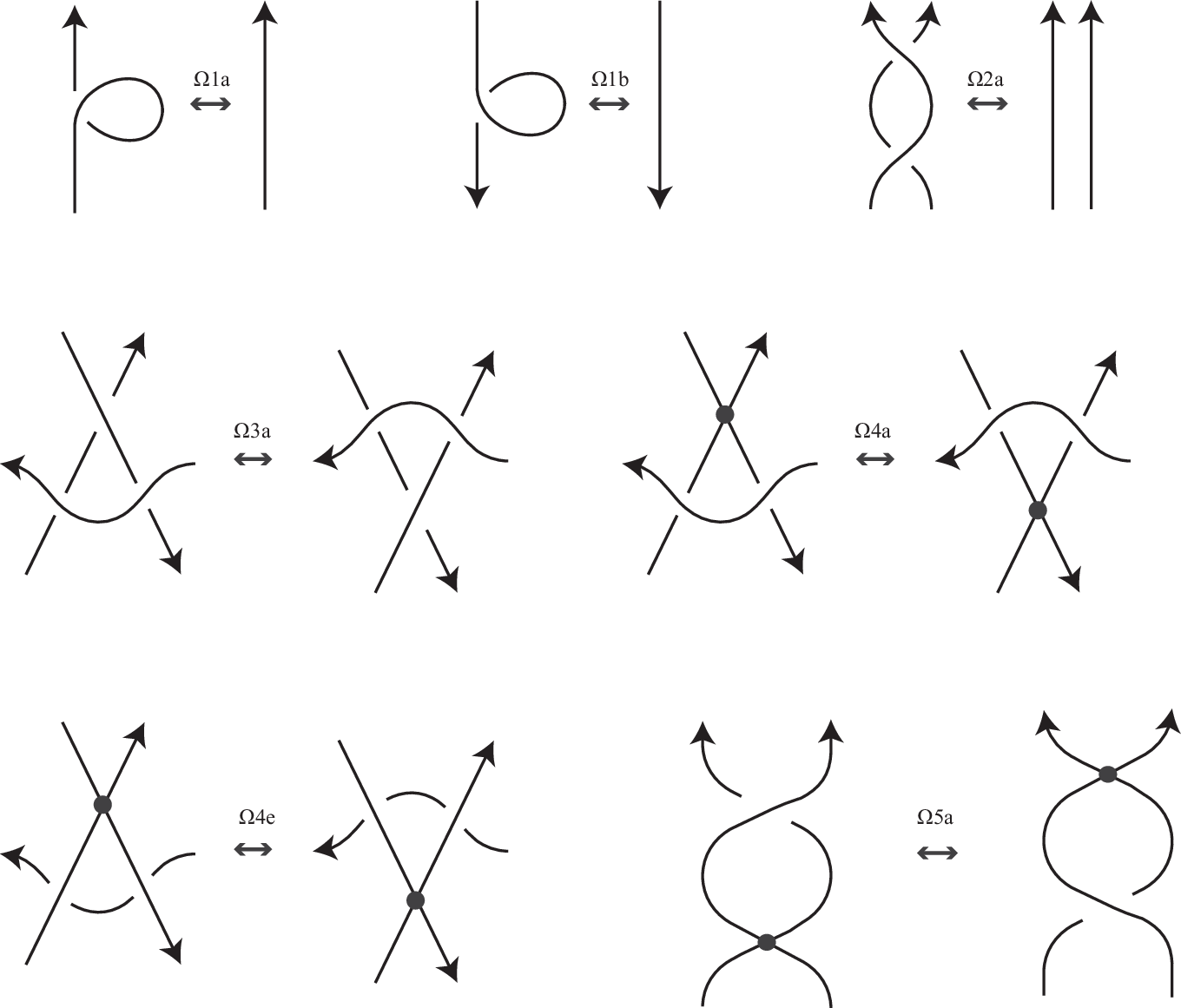}
  \caption{Generating Set of Oriented Singular Reidemeister Moves}
  \label{GSOOSRM}
\end{figure}\\
\begin{definition}\cite{BEHY} \label{Def2.2}
Let $(\mathnormal{X}, *)$ be a quandle. Let $\mathbf{R_1}$ and $\mathbf{R_2}$ be two maps from $\mathnormal{X} \times \mathnormal{X}$ to $\mathnormal{X}$. The quadruple $(\mathnormal{X}, *, \mathbf{R_1}, \mathbf{R_2})$ is called an {\it oriented singquandle} if the following axioms are satisfied:
\[ 
\mathbf{R_1}(x \bar{*} y, z) * y = \mathbf{R_1}(x, z * y), \tag{2.2.1} \label{eq:2.2.1} \]
\[
\mathbf{R_2}(x \bar{*} y, z) = \mathbf{R_2}(x, z * y) \bar{*} y, \tag{2.2.2} \label{eq:2.2.2}
\]
\[
(y \bar{*} \mathbf{R_1}(x, z)) * x = (y * \mathbf{R_2}(x, z)) \bar{*} z, \tag{2.2.3} \label{eq:2.2.3} 
\]
\[
\mathbf{R_2}(x, y) = \mathbf{R_1}(y, x * y), \tag{2.2.4} \label{eq:2.2.4} 
\]
\[
\mathbf{R_1}(x, y) * \mathbf{R_2}(x, y) = \mathbf{R_2}(y, x * y). \tag{2.2.5} \label{eq:2.2.5} 
\]
\end{definition}
\par All the axioms of the above definition \ref{Def2.2} are derived from the generating set of oriented singular Reidemeister moves \ref{GSOOSRM} after coloring its arcs by the elements of the quandle $\mathnormal{X}$.
\par Some typical examples of oriented singquandles are given below:
\begin{itemize}
    \item For a quandle $(X,*)$ with $X=G$ as a non-abelian multiplicative group and $x * y= y^{-1}xy$, the quadruple $(\mathnormal{X}, *, \mathbf{R_1}, \mathbf{R_2})$ forms an oriented singquandle if and only if the following conditions are satisfied:
    \begin{enumerate}
        \item $y^{-1} \mathbf{R_1}(yxy^{-1}, z)y = \mathbf{R_1}(x, y^{-1}zy)$,
        \item $\mathbf{R_2}(yxy^{-1}, z) = y\mathbf{R_2}(x, y^{-1}zy)y^{-1}$,
        \item $x^{-1}\mathbf{R_1}(x, z)y[\mathbf{R_1}(x, z)]^{-1}x = z[\mathbf{R_2}(x, z)]^{-1}y[\mathbf{R_2}(x, z)]z^{-1}$,
        \item $\mathbf{R_2}(x, y) = \mathbf{R_1}(y, y^{-1})xy$,
        \item $[\mathbf{R_2}(x, y)]^{-1}\mathbf{R_1}(x, y)\mathbf{R_2}(x, y) = \mathbf{R_2}(y, y^{-1}xy)$
    \end{enumerate}
    \item For a quandle $(X,*)$ where $X$ is a group $G$ and $x*y= y^{-1}xy$, then for $n \ge 1$ the quadruple $(\mathnormal{X}, *, \mathbf{R_1}, \mathbf{R_2})$ forms an oriented singquandle if:
    \begin{enumerate}
        \item $\mathbf{R_1}(x, y) = x(xy^{-1})^{n}$ and $ \mathbf{R_2}(x, y) = y(x^{-1}y)^{n}$,
        \item $\mathbf{R_1}(x, y) = (xy^{-1})^{n}x$ and $ \mathbf{R_2}(x, y) = (x^{-1}y)^{n}y$,
        \item $\mathbf{R_2}(x, y) = \mathbf{R_2}(xy^{-1}x, x)[\mathbf{R_1}(xy^{-1}x, x)]^{-1} \mathbf{R_2}(xy^{-1}x, x)$,
        \item $\mathbf{R_1}(x, y) = x(yx^{-1})^{n+1}$ and $ \mathbf{R_2}(x, y) = x(y^{-1}x)^{n}$.
    \end{enumerate}
\end{itemize}

\begin{definition}\cite{CCE1} \label{Deff2.2'}
    Let $(X, *_1, *_2, R_1, R_2)$ be an oriented singquandle. A non-empty subset $S$ of $X$ is called an oriented sub-singquandle of $X$ 
    if $(S, *_1, *_2, R_1, R_2)$ is an oriented singquandle.
\end{definition}

\begin{lemma}\cite{CCE1} \label{lem2.1}
    Let $(X, *_1, *_2, R_1, R_2)$ be an oriented singquandle. A non-empty subset $S$ of $X$ is an oriented sub-singquandle of $X$ if and only if $S$ is closed under the operations $*_1, *_2, R_1$ and $R_2$. 
\end{lemma}

\begin{definition}\label{Def2.3}
A classical $n$-link $\mathnormal{L}$ in $\mathbb{R}^{3}$ with its each component either labelled by $``1"$ or $``2"$ is called a dichromatic link with $n$ components and is represented by dichromatic link diagrams $\mathnormal{D}$ in $\mathbb{R}^{2}$.
\end{definition}
A dichromatic link $\mathnormal{L}$ is said to be oriented if its every component is oriented. For example, see Figure~\ref{OdcLinks} which shows two oriented dichromatic links with two components.
\begin{figure}[h]
		\vspace{0.5cm}
		\includegraphics[width=0.65\textwidth]{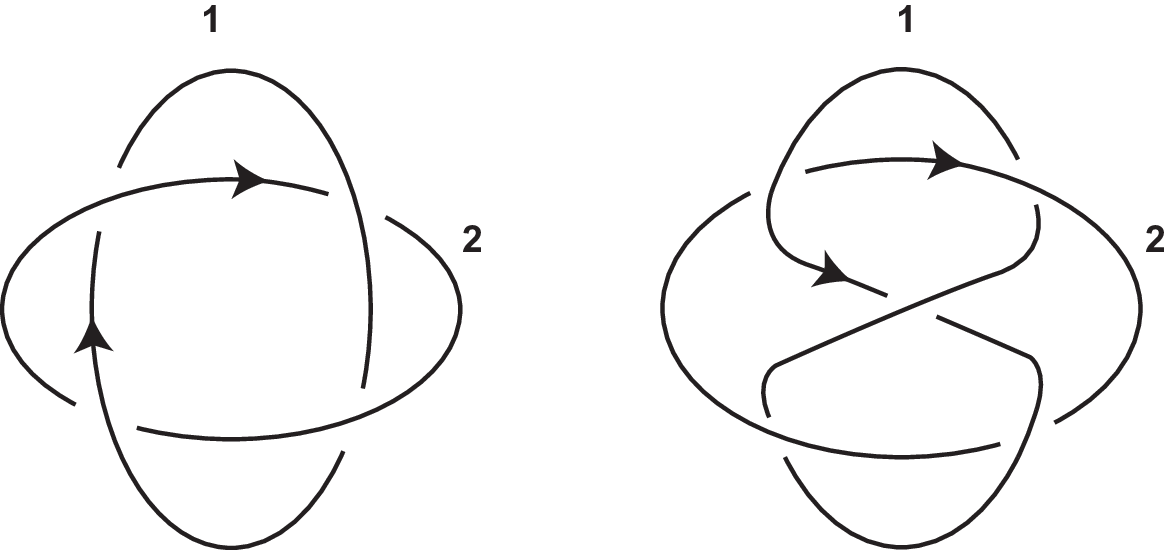}
	\caption{Oriented Dichromatic Links}
	\label{OdcLinks}
\end{figure}
\par After applying finite number of generalized Reidemeister moves of oriented dichromatic links in Figure \ref{OdcMoves}, if two oriented dichromatic link diagrams $\mathnormal{D_1}$ and $\mathnormal{D_2}$ representing the oriented dichromatic links $\mathnormal{L_1}$ and $\mathnormal{L_2}$ respectively, can be obtained from each other, then the links $\mathnormal{L_1}$ and $\mathnormal{L_2}$ are called isotopy equivalent.
\begin{figure}[h]
\centering
\includegraphics[width=0.9\textwidth]{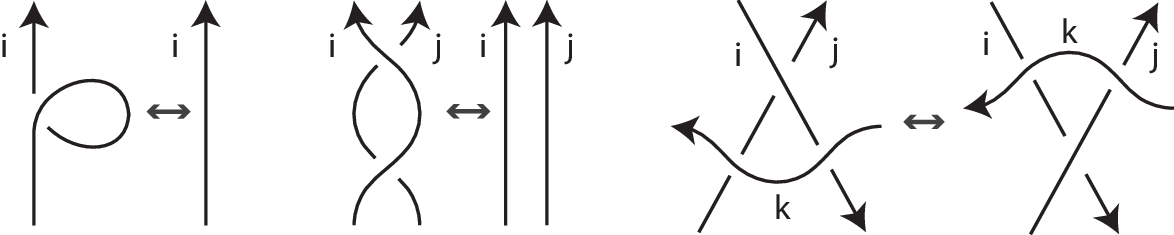}
	\caption{Generalized Reidemeister Moves for Oriented Dichromatic Links}
	\label{OdcMoves}
\end{figure}
\section{Oriented Dichromatic Singular Links}\label{ODSL}
In this section we define the analog of the links defined in 
\cite{SEA}. Thus we have the following definition.
\begin{definition}\label{Def3.1}
An oriented singular link $\mathnormal{L}$ in $\mathbb{R}^{3}$ whose each component is colored (labelled) by either $``1"$ or $``2"$ is called an {\it oriented dichromatic singular link}.
\end{definition}
An oriented dichromatic singular link $\mathnormal{L}$ in $\mathbb{R}^{3}$ is represented by an oriented dichromatic singular link diagram $\mathnormal{D}$ in $\mathbb{R}^{2}$ in which each component is labelled by $``1"$ or $``2"$. Figure \ref{OsdLinks} shows two examples of oriented dichromatic singular link diagrams.
\begin{figure}[h]
\includegraphics[width=0.7\textwidth]{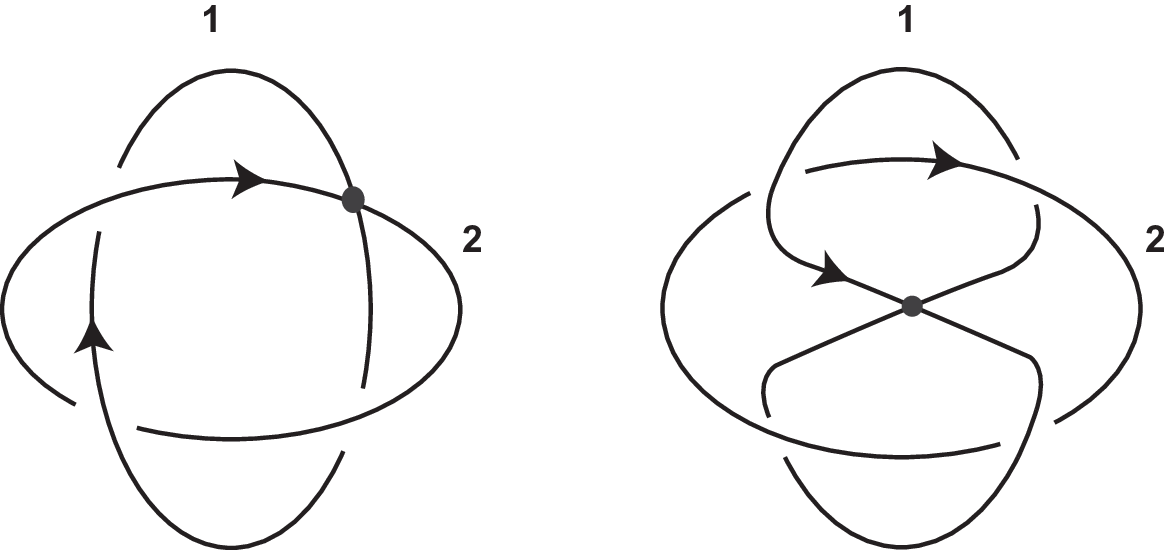}
	\caption{Oriented Dichromatic Singular Links}
	\label{OsdLinks}
\end{figure}
\par Two oriented dichromatic singular links $\mathnormal{L_1}$ and $\mathnormal{L_2}$ in $\mathbb{R}^{3}$ are said to be {\it ambient isotopic} if there exists an orientation preserving self homeomorphism $h: \mathbb{R}^{3} \rightarrow \mathbb{R}^{3}$ that takes one link to the other and preserves the singularities as well as the labels $``1", ``2"$ such that $h(\mathnormal{L_1})= \mathnormal{L_2}$. Thus two oriented dichromatic singular links $\mathnormal{L_1}$ and $\mathnormal{L_2}$ are equivalent if one can be obtained from the other by a finite sequence of generalized oriented dichromatic singular Reidemeister moves preserving the label of each component as shown in the Figure \ref{OdsRMoves}. Let $\mathnormal{D_1}$ and $\mathnormal{D_2}$ be two oriented dichromatic singular link diagrams in $\mathbb{R}^{2}$ representing $\mathnormal{L_1}$ and $\mathnormal{L_2}$, respectively. Then $\mathnormal{L_1}$ and $\mathnormal{L_2}$ are equivalent if and only if $\mathnormal{D_1}$ and $\mathnormal{D_2}$ can be transformed into each other by a finite sequence of generalized oriented dichromatic singular Reidemeister moves shown in the following Figure \ref{OdsRMoves} where ${i, j, k} \in \{1, 2\}$. 
\par An oriented dichromatic singular link with $n$ components is called as an $n$-component oriented dichromatic singular link. Thus an $n$-component oriented dichromatic singular link in $\mathbb R^3$ can be defined as $L=K_1\cup \cdots \cup K_n$. Taking $n=2$, we obtain $2$-component oriented dichromatic singular links. Some $2$-component oriented dichromatic singular link diagrams (see p 814 of \cite{Oyamaguchi}) are shown in Figure \ref{TOSDLinks}. 
\begin{figure}[htbp]
\centering
\includegraphics[width=0.8\textwidth]{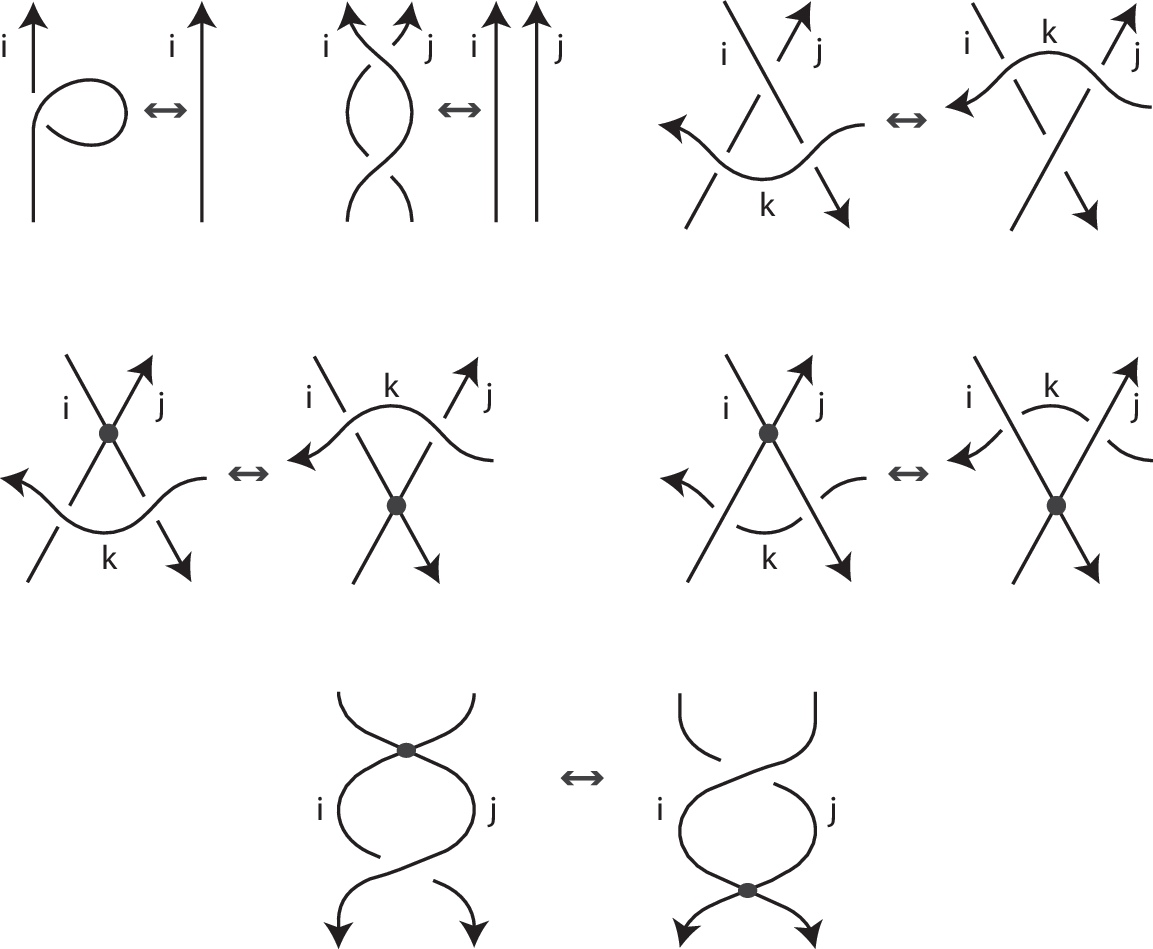}
	\caption{Regular Oriented Dichromatic Reidemeister Moves $RI, RII$ and $RIII$ on the top and Oriented Dichromatic Singular Reidemeister Moves $RIVa, RIVb$ and $RV$ in the middle and on the bottom.}
	\label{OdsRMoves}
\end{figure}
\vspace{0.6cm}
\begin{proposition}\label{Prop3.1}
Let $\mathnormal{L_1}$ and $\mathnormal{L_2}$ be two oriented dichromatic singular links in $\mathbb{R}^{3}$ and let $\mathnormal{D_1}$ and $\mathnormal{D_2}$ be two oriented dichromatic singular link diagrams in $\mathbb{R}^{2}$ representing $\mathnormal{L_1}$ and $\mathnormal{L_2}$, respectively. Then $\mathnormal{L_1}$ and $\mathnormal{L_2}$ are equivalent if and only if $\mathnormal{D_1}$ and $\mathnormal{D_2}$ are transformed into each other by a finite sequence of generalized oriented Reidemeister moves for oriented dichromatic singular links which preserve the orientation, singularities and the label of each component as shown in the Fig. \ref{OdsRMoves} where ${i, j, k} \in \{1, 2\}$ and ambient isotopies of $\mathbb{R}^{2}$.
\end{proposition}


\section{Oriented Disingquandles}\label{GFOSQ}
In this section we introduce an analog of the notion of $\mathbb{Z}_2$-Family of 
Singquandles defined in section $4$ of \cite{SEA}. The analog contains an extra information of orientation. Thus we have the following definition.
\begin{definition}\label{Def4.1}
Let $X$ be a set equipped with two binary operations $*_1$ and $*_2$ such that both $(X,*_1), (X,*_2)$ are quandles.  Let $\mathbf{R_1}, \mathbf{R_2}$ be two maps from $\mathnormal{X} \times \mathnormal{X}$ to $\mathnormal{X}$ such that the quadruples $(\mathnormal{X}, *_1, \mathbf{R_1}, \mathbf{R_2})$ and $(\mathnormal{X}, *_2, \mathbf{R_1}, \mathbf{R_2})$ are oriented singquandles. Then the quintuple $(\mathnormal{X}, *_1, *_2, \mathbf{R_1}, \mathbf{R_2})$ is called an {\it oriented disingquandle} 
if the following axioms are satisfied
\[
(y \bar{*}_{1} \mathbf{R_1}(x, z)) *_{2} x = (y *_{2} \mathbf{R_2}(x, z)) \bar{*}_{1} z, \tag{4.1.1} \label{eq:4.1.1}
\]
\[
(y \bar{*}_{2} \mathbf{R_1}(x, z)) *_{1} x = (y *_{1} \mathbf{R_2}(x, z)) \bar{*}_{2} z, \tag{4.1.2} \label{eq:4.1.2}
\]
\[
\mathbf{R_1}(x, y) *_{1} \mathbf{R_2}(x, y) = \mathbf{R_2}(y, x *_{2} y), \tag{4.1.3} \label{eq:4.1.3} 
\]
\[
\mathbf{R_1}(x, y) *_{2} \mathbf{R_2}(x, y) = \mathbf{R_2}(y, x *_{1} y), \tag{4.1.4} \label{eq:4.1.4} 
\]
\end{definition}
In other words an oriented disingquandle is a pair of quandles with binary operations $*_1$, $*_2$ and two additional maps $\mathbf{R_1}$ and $\mathbf{R_2}$ such that the singquandle axioms $(\ref{eq:2.2.1}) - (\ref{eq:2.2.5})$ as well as the additional axioms $(\ref{eq:4.1.1}) - (\ref{eq:4.1.4})$ are all satisfied simultaneously. Thus it follows from this definition that every  oriented disingquandle is a quandle.
\noindent
\begin{remark}
    Notice that in this definition, the fact that $(\mathnormal{X}, *_1, \mathbf{R_1}, \mathbf{R_2})$ is an oriented singquandle gives, by equation~(\ref{eq:2.2.5}) of Definition~\ref{Def2.2}, that $\mathbf{R_1}(x, y) *_1 \mathbf{R_2}(x, y) = \mathbf{R_2}(y, x *_1 y)$. Now equation~(\ref{eq:4.1.3}) gives $\mathbf{R_1}(x, y) *_{1} \mathbf{R_2}(x, y) = \mathbf{R_2}(y, x *_{2} y)$ and thus one obtains that  
    
\[
R_2(y,x*_1y)=R_2(y,x*_2y).
\]
\end{remark}

The above axioms of an oriented disingquandle come from the generalized oriented dichromatic singular Reidemeister moves shown in Figure~\ref{OdsRMoves} when we take the coloring rule shown in Figure~\ref{OdsLinkColoring} under consideration.
\begin{figure}[h]
		\includegraphics[width=0.8\textwidth]{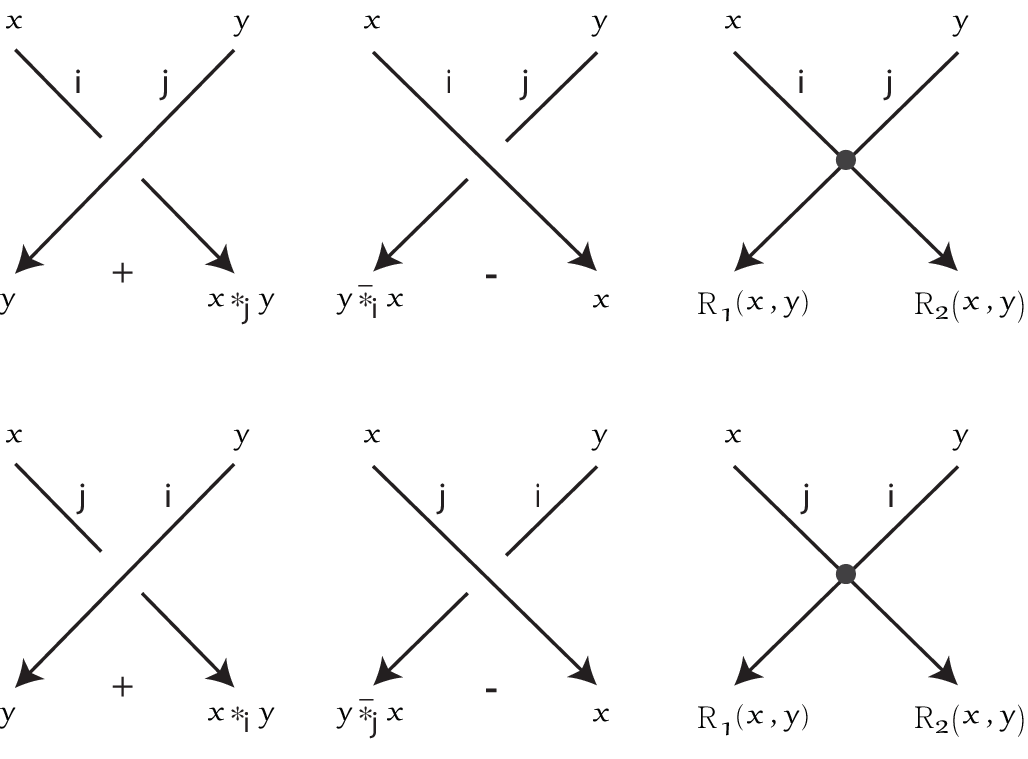}
	\caption{Coloring by an oriented disingquandle}
	\label{OdsLinkColoring}
\end{figure}
\par The following lemma is motivated by the above definition \ref{Def4.1}.
\begin{lemma}\label{lem4.1}
The set of colorings of an oriented dichromatic singular link by an oriented disingquandle does not change by the generalized oriented dichromatic singular Reidemeister moves shown in Figure \ref{OdsRMoves}.
\end{lemma}

\begin{proof}
 As in the case of oriented classical and singular knot theories, there is one to one correspondence between colorings before and after each of the oriented dichromatic singular Reidemeister moves.  The invariance follows directly from the equations~\ref{eq:4.1.1}, \ref{eq:4.1.2}, \ref{eq:4.1.3} and \ref{eq:4.1.4} given in Definition~\ref{Def4.1}.
\end{proof}

 Now we need examples of oriented disingquandles in order to compute the coloring invariant.  But first we have the following lemma.
 
\begin{lemma} \label{lemma4.4}
    Let $n>2$ be an integer.  Consider the quandle operation on $\mathbb{Z}_n$ given by $x*y=ax+(1-a)y$, $x \bar{*} y=a^{-1}x+(1-a^{-1})y$ and assume that $a$ is an invertible such that $a^2 \neq 1$.  Assume that $R_1(x,y) =\alpha +\beta x +\gamma y +\lambda x^2+\mu y^2+ \delta xy.$  Let $R_2$ be given by $R_2(x,y)=R_1(y,x*y)$.  Then the conditions $(1-a)\alpha=(1-a)[1-\beta - \gamma]=(1-a)\lambda=(1-a)\mu=(1-a)\delta=0$ 
    imply that $(\mathbb{Z}_n,*,R_1,R_2)$ is an oriented singquandle.
\end{lemma}
\begin{proof}
Substituting $R_1$, $R_2$, $*$ and $\bar{*}$ into equation (\ref{eq:2.2.1}) gives the following conditions
\[
(1-a)[1-\beta - \gamma]=(1-a)\lambda=(1-a)\mu=(1-a)\delta=0.
\]
Using these conditions one checks that both equations (\ref{eq:2.2.2}) and (\ref{eq:2.2.3}) hold. Now equation (\ref{eq:2.2.4}) holds by definition.  By expanding equation (\ref{eq:2.2.5}) one obtains a tautology.
\end{proof}

\begin{corollary}\label{Cor}
    
    With the same hypotheses as the previous Lemma~\ref{lemma4.4} and assuming $*_1=*_2$ then the quintuple $(\mathbb{Z}_n, *_1, *_2, \mathbf{R_1}, \mathbf{R_2})$ is an {\it oriented disingquandle.} 
\end{corollary}

Using Lemma~\ref{lemma4.4} and Corollary~\ref{Cor}, one can obtain many examples of oriented disingquandles.  Here we give a few  explicit examples.
\begin{example}\label{Uno}
 Consider $X=\mathbb{Z}_{10}$ with $x*_1y=x*_2y=x*y=3x-2y$, $x \bar{*}y=7x-6y$, $R_1(x,y)=5+x+5y+5x^2+5y^2+5xy$ and $R_2(x,y)=R_1(y,x*y)=6x+5xy$, then the quintuple $(\mathbb{Z}_{10}, *_1, *_2, \mathbf{R_1}, \mathbf{R_2})$ is an {\it oriented disingquandle}.  Notice that in this example $1-a=-2$.  We chose $\gamma=5$, $\beta=1$ and $\alpha=\lambda=\mu=\delta=5$.
\end{example}

\begin{example}
 Lemma~\ref{lemma4.4} and Corollary~\ref{Cor} give the following disingquandle  $(\mathbb{Z}_{10},*_1,*_2, R_1,R_2)$, where $x*_1y=x*_2y=x*y = 3x-2y$ and $x \bar{*}y=7x-6y$.  Let $R_1(x,y) =4x+2y+ 5xy$.  Thus $R_2(x,y)= 6x+5xy$.  
\end{example}

\begin{example}\label{Z30}
     Consider $X=\mathbb{Z}_{30}$ with $x*_1y=x*_2y=x*y=13x-12y$, $x \bar{*}y=7x-6y$.
     $R_1(x,y)=5-9x+5y+10x^2+15y^2+20xy$ and $R_2(x,y)=R_1(y,x*y)=5+5x+21y+15x^2+10y^2+20xy,$ then the quintuple $(\mathbb{Z}_{30}, *_1, *_2, \mathbf{R_1}, \mathbf{R_2})$ is an {\it oriented disingquandle}.  
\end{example}

    \begin{example}\label{Z60}
     Consider $X=\mathbb{Z}_{60}$ with $x*_1y=x*_2y=x*y=7x-6y$, $x \bar{*}y=-17x+18y$.
     Since $1-a=-6$, we choose $\gamma=5$, $R_1(x,y)=10+6x+5y+10x^2+20y^2+30xy$ and $R_2(x,y)=R_1(y,x*y)=10+35x-24y+20x^2+10y^2,$ then the quintuple $(\mathbb{Z}_{60}, *_1, *_2, \mathbf{R_1}, \mathbf{R_2})$ is an {\it oriented disingquandle}.  
\end{example}

Since a quandle can be represented by a matrix, so the same representation is possible for an oriented disingquandle. Let $X=\{x_1,x_2,\ldots,x_n\}$ be a finite oriented disingquandle with two operations $*_1$, $*_2$ and the additional maps $\mathbf{R_1}$ and $\mathbf{R_2}$. The {\it presentation matrix} of the oriented disingquandle $(X,*_1, *_2, R_1, R_2)$, denoted by $M_X$, is defined to be the block matrix:
	$$M_X=\left[\begin{array}{c|c|c|c}
	M^1&M^2&M^3&M^4
	\end{array} \right],$$ where $M^1=(m^1_{ij})_{1\leq i,j\leq n}$, $M^2=(m^2_{ij})_{1\leq i,j\leq n}$, $M^3=(m^3_{ij})_{1\leq i,j\leq n}$ and $M^4=(m^4_{ij})_{1\leq i,j\leq n}$ are $n\times n$ matrices with the entries in $X$ given by $$m^k_{ij}=\left\{
	\begin{array}{ll}
	{x_i} *_1 {x_j} & \hbox{for $k=1$,} \\
	{x_i} *_2 {x_j}& \hbox{for $k=2$,}\\
        R_1(x_i, x_j)& \hbox{for $k=3$,}\\
        R_2(x_i, x_j)& \hbox{for $k=4$.}
	\end{array}%
	\right.$$  
 \begin{example}\label{example4.8}
 In this example, we give the matrix presentation of the oriented disingquandle of order $10$ given in Example~\ref{Uno}.  The matrices $M^1, M^2, M^3$ and $M^4$ below are respectively the Cayley tables of the operations $*_1$, $*_2$, and the maps $R_1$ and $R_2$. 

\begin{equation*}
	M_X=\left[\begin{array}{c|c|c|c} 
	M^1 & M^2 & M^3 & M^4
	\end{array} \right]. 
	\end{equation*}

 Where
 
\begin{equation*}
M^1= \begin{bmatrix}
0 & 8 & 6 & 4 & 2 & 0 & 8 & 6 & 4 & 2 \\
3 & 1 & 9 & 7 & 5 & 3 & 1 & 9 & 7 & 5 \\
6 & 4 & 2 & 0 & 8 & 6 & 4 & 2 & 0 & 8 \\
9 & 7 & 5 & 3 & 1 & 9 & 7 & 5 & 3 & 1 \\
2 & 0 & 8 & 6 & 4 & 2 & 0 & 8 & 6 & 4 \\
5 & 3 & 1 & 9 & 7 & 5 & 3 & 1 & 9 & 7 \\
8 & 6 & 4 & 2 & 0 & 8 & 6 & 4 & 2 & 0 \\
1 & 9 & 7 & 5 & 3 & 1 & 9 & 7 & 5 & 3 \\
4 & 2 & 0 & 8 & 6 & 4 & 2 & 0 & 8 & 6 \\
7 & 5 & 3 & 1 & 9 & 7 & 5 & 3 & 1 & 9
\end{bmatrix}, \hspace{4em}
M^2= \begin{bmatrix}
0 & 8 & 6 & 4 & 2 & 0 & 8 & 6 & 4 & 2 \\
3 & 1 & 9 & 7 & 5 & 3 & 1 & 9 & 7 & 5 \\
6 & 4 & 2 & 0 & 8 & 6 & 4 & 2 & 0 & 8 \\
9 & 7 & 5 & 3 & 1 & 9 & 7 & 5 & 3 & 1 \\
2 & 0 & 8 & 6 & 4 & 2 & 0 & 8 & 6 & 4 \\
5 & 3 & 1 & 9 & 7 & 5 & 3 & 1 & 9 & 7 \\
8 & 6 & 4 & 2 & 0 & 8 & 6 & 4 & 2 & 0 \\
1 & 9 & 7 & 5 & 3 & 1 & 9 & 7 & 5 & 3 \\
4 & 2 & 0 & 8 & 6 & 4 & 2 & 0 & 8 & 6 \\
7 & 5 & 3 & 1 & 9 & 7 & 5 & 3 & 1 & 9
\end{bmatrix},
\end{equation*}

\begin{equation*}
M^3= \begin{bmatrix}
0 & 2 & 4 & 6 & 8 & 0 & 2 & 4 & 6 & 8 \\
4 & 1 & 8 & 5 & 2 & 9 & 6 & 3 & 0 & 7 \\
8 & 0 & 2 & 4 & 6 & 8 & 0 & 2 & 4 & 6 \\
2 & 9 & 6 & 3 & 0 & 7 & 4 & 1 & 8 & 5 \\
6 & 8 & 0 & 2 & 4 & 6 & 8 & 0 & 2 & 4 \\
0 & 7 & 4 & 1 & 8 & 5 & 2 & 9 & 6 & 3 \\
4 & 6 & 8 & 0 & 2 & 4 & 6 & 8 & 0 & 2 \\
8 & 5 & 2 & 9 & 6 & 3 & 0 & 7 & 4 & 1 \\
2 & 4 & 6 & 8 & 0 & 2 & 4 & 6 & 8 & 0 \\
6 & 3 & 0 & 7 & 4 & 1 & 8 & 5 & 2 & 9
\end{bmatrix},  \hspace{4em}
M^4= \begin{bmatrix}
0 & 0 & 0 & 0 & 0 & 0 & 0 & 0 & 0 & 0 \\
6 & 1 & 6 & 1 & 6 & 1 & 6 & 1 & 6 & 1 \\
2 & 2 & 2 & 2 & 2 & 2 & 2 & 2 & 2 & 2 \\
8 & 3 & 8 & 3 & 8 & 3 & 8 & 3 & 8 & 3 \\
4 & 4 & 4 & 4 & 4 & 4 & 4 & 4 & 4 & 4 \\
0 & 5 & 0 & 5 & 0 & 5 & 0 & 5 & 0 & 5 \\
6 & 6 & 6 & 6 & 6 & 6 & 6 & 6 & 6 & 6 \\
2 & 7 & 2 & 7 & 2 & 7 & 2 & 7 & 2 & 7 \\
8 & 8 & 8 & 8 & 8 & 8 & 8 & 8 & 8 & 8 \\
4 & 9 & 4 & 9 & 4 & 9 & 4 & 9 & 4 & 9
\end{bmatrix}. 
\end{equation*}
 
\end{example}

\begin{definition}\label{Def4.2}
Let $(X, *_1, *_2, R_1, R_2)$ and $(Y, *'_1, *'_2, R'_1, R'_2)$ be two oriented disingquandles. A map $f:X\to Y$ is called a {\it homomorphism} if $f(x*_1y)=f(x)*'_1f(y)$, $f(x*_2y)=f(x)*'_2f(y)$, $f(R_1(x,y))=R'_1(f(x),f(y))$ and $f(R_2(x,y))=R'_1(f(x),f(y))$ for all $x,y\in X.$\\
\par If a homomorphism is bijective, then it is called an {\it isomorphism}. Two oriented disingquandles $(X, *_1, *_2, R_1, R_2)$ and $(Y, *'_1, *'_2, R'_1, R'_2)$ are said to be {\it isomorphic} if there exists an isomorphism between them.
\end{definition}
\begin{definition}\label{Deff4.3}
    Let $(X, *_1, *_2, R_1, R_2)$ be a disingquandle. A non-empty subset $S$ of $X$ is called a subdisingquandle of $X$ if and only if $S$ is itself a disingquandle.
\end{definition}


\begin{lemma}
    The image, $\textit{Im}(f)$, of any oriented disingquandle homomorphism $f \colon (X, *_1, *_2, R_1, R_2) \rightarrow (Y, \triangleright_1, \triangleright_2, R'_1, R'_2)$ is a sub-disingquandle. 
\end{lemma}

\begin{proof}
Let $f \colon (X, *_1, *_2, R_1, R_2) \rightarrow (Y, \triangleright_1, \triangleright_2, R'_1, R'_2)$ be a disingquandle homomorphism.  Definition~\ref{Def4.2} implies that $\textit{Im}(f)$ is closed under $\triangleright_1$, $\triangleright_2$, $R'_1$ and $R'_2$.  Since the identities of both Definition~\ref{Def2.2} and Definition~\ref{Def4.1} are satisfied in $Y$ then they are automatically satisfied in $\textit{Im}(f)$.  This ends the proof. 
\end{proof}

\begin{theorem}\label{Theorem4.1}
    Let $(X, *_1, *_2, R_1, R_2)$ be a oriented disingquandle. A non-empty subset $S$ of $X$ is an oriented sub-disingquandle of $X$ if and only if $S$ is closed under the operations $*_1, *_2, R_1$ and $R_2$.
\end{theorem}
\begin{proof}
    Let $S$ be a sub-disingquandle of the oriented disingquandle $(X, *_1, *_2, R_1, R_2)$, then by definition \ref{Def4.1} it follows that $(S, *_1, R_1, R_2)$ and $(S, *_2, R_1, R_2)$ are oriented singquandles. Therefore by the Lemma \ref{lem2.1}, it follows that $S$ is closed under the operations $*_1, *_2, R_1$ and $R_2$.
    \par Conversely suppose that $S$ is closed under the operations $*_1, *_2, R_1$ and $R_2$. Then by the Lemma \ref{lem2.1}, it follows that $(S, *_1, R_1, R_2)$ and $(S, *_2, R_1, R_2)$ are oriented singquandles. Since the identities of Definition~\ref{Def4.1} are satisfied in $X$ then they are automatically satisfied in $S$.

\end{proof}


\section{Invariants for Oriented Dichromatic Singular Links}\label{CIUSDL}

Let $D$ be an oriented dichromatic singular link diagram $\mathbb R^2$ corresponding to an oriented dichromatic singular link $\mathnormal{L}$ in $\mathbb R^3$ and let $\mathcal A(D)$ denote the set of all arcs of $D$. Let $(X,*_1, *_2, \mathbf{R_1}, \mathbf{R_2})$ be an oriented disingquandle.  An {\it oriented disingquandle coloring} of $D$ by $X$, or simply {\it oriented disingquandle $X$-coloring} of $D$, is a map $\mathcal{C}: \mathcal A(D) \rightarrow X$ such that at every classical and singular crossing, the relations depicted in Figure \ref{OdsLinkColoring} hold. The oriented disingquandle element $\mathcal{C}(s)$ is called a {\it color} of the arc $s$ and the pair $(D, \mathcal C)$ is called the {\it $X$-colored oriented dichromatic singular link diagram by $\mathcal C$}. The set of all oriented disingquandle $X$-colorings of $D$ is denoted by ${\rm Col}^{odsq}_X(D)$. Thus we have the following results:

\begin{lemma}\label{lem5.1} 
Let $D$ and $D'$ be two oriented dichromatic singular link diagrams in $\mathbb R^2$ that can be transformed into each other by oriented generalized dichromatic singular Reidemeister moves as shown in Figure \ref{OdsRMoves}. Then for any finite oriented disingquandle $X$, there is a one-to-one correspondence between ${\rm Col}^{odsq}_X(D)$ and ${\rm Col}^{odsq}_X(D')$. 
\end{lemma}

\begin{proof}
    There is one-to-one correspondence between colorings before and after each of the 
    moves in Figure~\ref{OdsRMoves}.  The invariance follows automatically from equations~(\ref{eq:2.2.1}), (\ref{eq:2.2.2}), (\ref{eq:2.2.3}), (\ref{eq:2.2.4}), (\ref{eq:2.2.5}), (\ref{eq:4.1.1}), (\ref{eq:4.1.2}), (\ref{eq:4.1.3}) and (\ref{eq:4.1.4}).
\end{proof}
\par In an $X$-colored oriented dichromatic singular link diagram $(D, \mathcal C)$, we treat the elements of an oriented disingquandle $X$ as labels for the arcs in $D$ with different operations at crossings as shown in Figure \ref{OdsLinkColoring}. Then it is clear from Lemma \ref{lem5.1} that the oriented disingquandle axioms of Definition \ref{Def4.1} are transcriptions of a generating set of oriented generalized oriented Reidemeister moves for oriented dichromatic singular links which are sufficient to generate any other oriented generalized dichromatic singular Reidemeister moves. That is, the axioms \ref{eq:4.1.1}, \ref{eq:4.1.2}, \ref{eq:4.1.3} and \ref{eq:4.1.4} come from the oriented generalized dichromatic singular Reidemeister move $RIVa$, $RIVb$ and $RV$ as seen in Figure \ref{OdsRMoves}.
	
\begin{theorem}\label{theorem5.1} 
Let $L$ be an oriented dichromatic singular link in $\mathbb R^3$ and let $D$ be a diagram of $L$. Then for any finite oriented disingquandle $X$, the cardinality $\sharp{\rm Col}^{odsq}_X(L)$ is an invariant of $L$.
\end{theorem}
	
\begin{proof}
		Let $D'$ be any other oriented dichromatic singular link diagram of $L$ obtained from $D$ by applying a finite number of oriented generalized dichromatic singular Reidemeister moves. Then it is evident from Lemma \ref{lem5.1} that $\sharp{\rm Col}^{odsq}_X(D')=\sharp{\rm Col}^{odsq}_X(D)$. This completes the proof.
	\end{proof}	
If $X$ is a finite oriented disingquandle, we call the cardinality $\sharp{\rm Col}^{odsq}_X(D)$ the {\it oriented disingquandle $X$-coloring number} or the {\it oriented disingquandle counting invariant} of $L$, and denote it by $\mathbb Z^{odsq}_X(L)$, i.e., $\mathbb Z^{odsq}_X(L)=\sharp{\rm Col}^{odsq}_X(D).$

\begin{example}
Let $\mathcal{L}$ be the set of the \emph{eighteen} link diagrams shown in  Figure~\ref{TOSDLinks}.  The goal of this example is to distinguish some of the links in the set $\mathcal{L}$.
 First, we use the oriented disingquandle $(\mathbb{Z}_{10},*_1,*_2, R_1,R_2)$ of Example~\ref{Uno}, where the operations $*_1=*_2=*$ such that  $x* y = 3x-2y$, $x \bar{*}y=7x-6y$, $R_1(x,y) =4x+2y+ 5xy$ and $R_2(x,y)= 6x+5xy$.   
Using the labelings of the $18$ links as given in Figure~\ref{TOSDLinks}, we compute the number of colorings $\mathbb Z^{odsq}_{10}(L)$ 
for each link $L$ in $\mathcal{L}$.  The computations were done by writing the set of equations obtained at each crossing of each link and solving the system of equations.   
The computations were simultaneously obtained by hand calculations and with the help of Python. 
\end{example}
\begin{table}[H]
\begin{center}
\begin{tabular}{|l|c|c||} \hline 
 $\mathbb Z^{odsq}_{10}(L)$ & Links\\  
     \hline
   $\quad \quad 10$ &   $4_1^2$, $5_1^2$, $5_2^2$, $6_1^2$, $6_2^2$, $6_3^2$,$6_5^2$,  $6_6^2$, $6_7^2$, $6_8^2$, $6_{9}^2$,,  $6_{11}^2$, $6_{12}^2$     \\
    \hline
    $\quad \quad 14$ &  $5_3^2$ \\
     \hline
   $\quad \quad 50$  & $3_1^2$, $6_4^2$, $6_{10}^2$   \\
   \hline
   $\quad \quad 75$  &   $1_1^2$   \\
    \hline
     
\end{tabular}\end{center}
\caption{Number of colorings of links of Figure~\ref{TOSDLinks} by $\mathbb{Z}_{10}$ }
\label{Table1}
\end{table}

Now using the oriented disingquandle given in Example~\ref{Z30} with  $X=\mathbb{Z}_{30}$ where the operations $x*_1y=x*_2y=x*y=13x-12y$, $x \bar{*}y=7x-6y$,
     $R_1(x,y)=5-9x+5y+10x^2+15y^2+20xy$ and $R_2(x,y)=R_1(y,x*y)=5+5x+21y+15x^2+10y^2+20xy$,   we get the following number of colorings for the link diagrams shown in the Table~\ref{TOSDLinks}. 

\begin{table}[h]
\begin{center}
\begin{tabular}{|l|c|c||} \hline 
 $\mathbb Z^{odsq}_{30}(L)$ & Links\\ 
 \hline
   $\quad \quad 0$  &  $6_{12}^2$    \\
     \hline
   $\quad \quad 30$ &  $3_1^2$, $4_1^2$, $5_1^2$, $5_2^2$,  $6_2^2$, $6_3^2$,$6_4^2$, $6_5^2$,  $6_6^2$, $6_7^2$, $6_8^2$, $6_{9}^2$, $6_{10}^2$,  $6_{11}^2$     \\
    \hline
    $\quad \quad 49$ &  $5_3^2$ \\
     \hline
   $\quad \quad 50$  & $1_1^2$     \\
   \hline
   $\quad \quad 150$  &  $6_1^2$   \\
    \hline

\end{tabular}\end{center}
\caption{Number of colorings of links of Figure~\ref{TOSDLinks} by $\mathbb{Z}_{30}$ }
\label{Table2}
\end{table}

Let $\Psi(L):= (\mathbb Z^{odsq}_{10}(L), \mathbb Z^{odsq}_{30}(L))$ 
denotes the pair of the number of colorings of a link $L$.  This is an invariant of oriented dichromatic singular Links.  For the link diagrams shown in Figure~\ref{TOSDLinks}, we use the two oriented disingquandles $\mathbb{Z}_{10} $ and $ \mathbb{Z}_{30}$ to obtain the invariant from Table~\ref{Table1} and Table~\ref{Table2}.  We then have the following table of pairs of colorings.





\begin{table}[h]

\begin{tabular}{|l|c|c||} \hline 
 $\Psi(L)$ & $L$\\ 
 \hline
   $(75, 50)$  &  $1_{1}^2$    \\
     \hline
   $(14,49)$ &  $5_3^2$ \\
    \hline
    $(10, 150)$ &  $6_1^2$ \\
     \hline
   $(50, 30)$  & $3_1^2$, $6_4^2$, $6_{10}^2$  \\
   \hline
   $(10, 30)$  &  $4_1^2$, $5_1^2$, $5_2^2$, $6_2^2$, $6_3^2$, $6_5^2$, $6_6^2$, $6_7^2$, $6_8^2$, $6_9^2$, $6_{11}^2$  \\
   \hline
    $(10, 0)$ &  $6_{12}^2$ \\
    \hline
    
\end{tabular}
\caption{Table of pairs of colorings of the link diagrams of Figure~\ref{TOSDLinks} by the two oriented disingquandles $\mathbb{Z}_{10}$ and $\mathbb{Z}_{30}$, respectively.}
\end{table}

\bigskip

\begin{figure}[ht]
\centering
\includegraphics[width=0.9\textwidth]{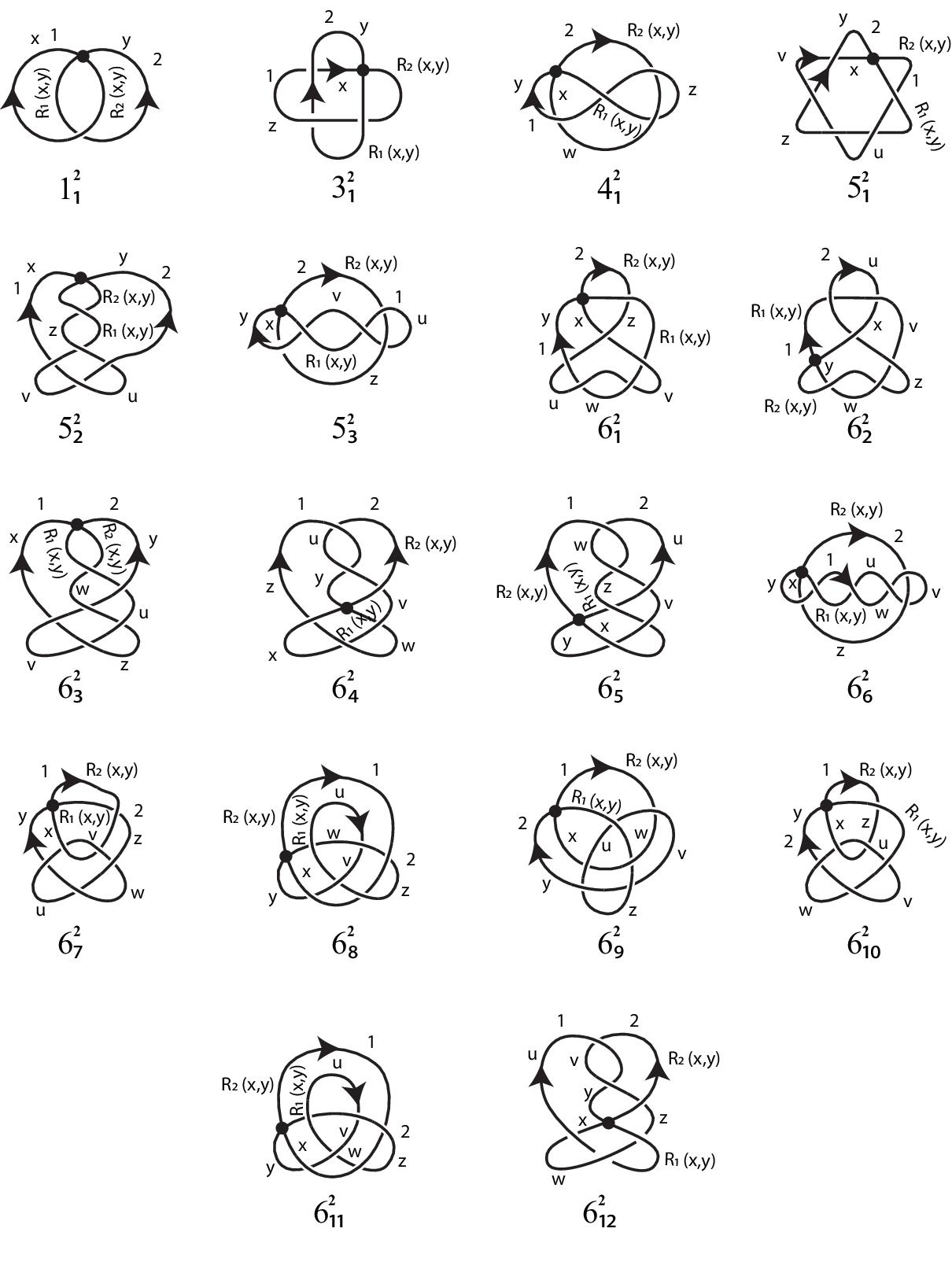}
	\caption{Table of oriented Dichromatic Singular Links}
	\label{TOSDLinks}
\end{figure}

\newpage
\section*{Acknowledgement} 
Mohamed Elhamdadi was partially supported by Simons Foundation collaboration grant 712462.

\end{document}